\documentclass[reqno]{amsart}
\usepackage{amssymb}
\usepackage{graphicx}

\usepackage[usenames, dvipsnames]{color}
\usepackage{verbatim}
\usepackage{mathrsfs}
\usepackage{bm}
\usepackage{cite}
\usepackage{esint}

\numberwithin{equation}{section}

\newtheorem{theorem}{Theorem}[section]

\newtheorem{lemma}[theorem]{Lemma}
\newtheorem{prop}[theorem]{Proposition}

\theoremstyle{definition}
\newtheorem{remark}[theorem]{Remark}

\theoremstyle{definition}

\theoremstyle{definition}

\makeatletter
\def\dashint{\operatorname%
{\,\,\text{\bf-}\kern-.98em\DOTSI\intop\ilimits@\!\!}}
\makeatother

\def\\det{\text{det}}

\def\.5{\frac{1}{2}}

\newcommand{\RN}[1]{%
  \textup{\uppercase\expandafter{\romannumeral#1}}%
}

\renewcommand{\epsilon}{\varepsilon}

\newcounter{marnote}

\begin{document}
\title[Gradient estimates for the insulated conductivity problem]{Gradient estimates for the insulated conductivity problem: the case of $m$-convex inclusions}

%\author[C.X. Miao]{Changxing Miao}
%\address[C.X. Miao]{Institute of Applied Physics and Computational Mathematics, P.O. Box 8009, Beijing, 100088, China.}
%\email{miao\_changxing@iapcm.ac.cn}

\author[Z.W. Zhao]{Zhiwen Zhao}

\address[Z.W. Zhao]{Beijing Computational Science Research Center, Beijing 100193, China.}
%\address{2. School of Mathematical Sciences, Beijing Normal University, Beijing 100875, China.}
\email{zwzhao365@163.com}

%\thanks{}
%\footnote{ }

\date{\today} % delete this line to display the current date

%%% BEGIN DOCUMENT

%\tableofcontents

\begin{abstract}
We consider an insulated conductivity model with two neighboring inclusions of $m$-convex shapes in $\mathbb{R}^{d}$ when $m\geq2$ and $d\geq3$. We establish the pointwise gradient estimates for the insulated conductivity problem and capture the gradient blow-up rate of order $\varepsilon^{-1/m+\beta}$ with $\beta=[-(d+m-3)+\sqrt{(d+m-3)^{2}+4(d-2)}]/(2m)\in(0,1/m)$, as the distance $\varepsilon$ between these two insulators tends to zero. In particular, the optimality of the blow-up rate is also demonstrated for a class of axisymmetric $m$-convex inclusions. 

\end{abstract}

\maketitle
%\date{}
%\maketitle
%{\bf Abstract}

%{Keywords}

%{\bf Mathematics Subject Classification(2010)} $35{\rm B}33 \cdot   35{\rm B}40 \cdot  35{\rm B}44$

\section{Introduction}

In this paper, we consider a bounded domain $D\subseteq\mathbb{R}^{d}\,(d\geq3)$ with $C^{2}$ boundary, which contains two $C^{2,\gamma}\,(0<\gamma<1)$ inclusions $D_{1}$ and $D_{2}$ with $\varepsilon$ apart. Moreover, these two inclusions are far away from the exterior boundary $\partial D$. Denote $\Omega:=D\setminus\overline{D_{1}\cup D_{2}}$. The insulated conductivity problem is modeled as follows:
\begin{align}\label{con002}
\begin{cases}
\Delta u=0,&\hbox{in}\;\Omega,\\
\frac{\partial u}{\partial\nu}=0,&\mathrm{on}\;\partial D_{i},\,i=1,2,\\
u=\varphi, &\mathrm{on}\;\partial D,
\end{cases}
\end{align}
where $\varphi\in C^{2}(\partial D)$ is a given boundary data, $\nu$ denotes the unit outer normal to the domain. The solution $u$ represents the voltage potential and its gradient $|\nabla u|$ is called the electric field. The electric field always exhibits high concentration in the thin gap between inclusions.  The interest of this paper lies in establishing the optimal gradient estimates for problem \eqref{con002} with two nearly touching $m$-convex inclusions in dimensions greater than two.

\subsection{Previous works}

Ammari et al.\cite{AKL2005,AKLLL2007} were the first to study the insulated conductivity problem and found that the optimal gradient blow-up rate is $\varepsilon^{-1/2}$ in dimension two. Bao, Li and Yin \cite{BLY2010} utilized a ``flipping" technique to obtain the upper bound on the gradient as follows:
\begin{align}\label{LYL90}
\|\nabla u\|_{L^{\infty}(\Omega)}\leq C\|\varphi\|_{C^{2}(\partial D)}\varepsilon^{-1/2},\quad\text{for any}\;d\geq2.
\end{align}
Subsequently, Yun \cite{Y2016} considered a pair of unit spheres in three dimensions and established the optimal gradient estimates only in the shortest segment between these two insulators, which revealed that the blow-up rate is $\varepsilon^{\frac{\sqrt{2}-2}{2}}$. Li and Yang \cite{LY2021,LY202102} further improved and extended the upper bound \eqref{LYL90} to
\begin{align}\label{TOP002}
\|\nabla u\|_{L^{\infty}(\Omega)}\leq C\|\varphi\|_{C^{2}(\partial D)}\varepsilon^{-1/m+\beta},\quad\text{with}\;d\geq3,\,m\geq2,
\end{align}
for two adjacent $m$-convex insulators, where $\beta>0$ is not explicit. The upper bound in \eqref{TOP002} also shows that the singularity of the electric field will weaken as the interfacial boundaries of insulators become flat. So it is essential and important to study the general $m$-convex insulators from the view of engineering, especially applying in the optimal shape design of insulated materials. The subsequent work \cite{W2021} completed by Weinkove gave an explicit $\beta(d)$, which sharpens the upper bound \eqref{TOP002} in the case of $m=2$ and $d\geq4$. Recently, Dong, Li and Yang \cite{DLY2021} established the optimal gradient estimates and obtained the explicit blow-up rate as follows:
\begin{align*}
|\nabla u|\sim\varepsilon^{\frac{\alpha-1}{2}},\quad\alpha=\frac{-(d-1)+\sqrt{(d-1)^{2}+4(d-2)}}{2},\quad\mathrm{for}\;m=2,\,d\geq3.
\end{align*}
This demonstrates that in the case of $m=2$ and $d=3$, the blow-up rate $\varepsilon^{\frac{\sqrt{2}-2}{2}}$ captured in \cite{Y2016} is also optimal in the whole matrix region, especially covering the shortest segment.

Problem \eqref{con002} is actually the limit equation of the following conductivity model with piecewise constant coefficients
\begin{align}\label{pro006}
\begin{cases}
\mathrm{div}(a_{k}(x)\nabla u_{k})=0,&\mathrm{in}\;D,\\
u_{k}=\varphi,&\mathrm{on}\;\partial D,
\end{cases}\quad a_{k}(x)=&
\begin{cases}
k\in(0,\infty),&\mathrm{in}\;D_{1}\cup D_{2},\\
1,&\mathrm{in}\;\Omega,
\end{cases}
\end{align}
with $\varphi\in C^{2}(\partial D)$, as the conductivity $k$ degenerates to zero. For problem \eqref{pro006}, Dong and Li \cite{DL2019} used Green's function method to capture the explicit dependence of the gradient on the conductivity $k$ and the distance $\varepsilon$ between two disks in two dimensions. This especially answers open problem $(b)$ proposed by Li and Vogelius in \cite{LV2000}. For more related work on the finite coefficients, we refer to \cite{BV2000,DZ2016,CEG2014,KL2019} for the elliptic equation and \cite{LN2003,BASL1999} for the elliptic systems, respectively. In particular, the problem of estimating the gradient of solutions in the presence of closely located inclusions was first posed in \cite{BASL1999}, which concerns a numerical investigation on the Lam\'{e} system arising from composites and motivates greatly the aforementioned and subsequent theoretical studies.

When the conductivity $k\rightarrow\infty$, the limit equation of \eqref{pro006} is called the perfect conductivity equation. There has been a long list of papers devoting to the study on the gradient estimates and asymptotics for the perfect conductivity problem. In the presence of strictly convex inclusions (that is, $m=2$), the blow-up rate of the concentrated field has been proved to be
\begin{align*}
\rho_{d}(\varepsilon)=&
\begin{cases}
\varepsilon^{-1/2},&\mathrm{if}\;d=2,\\
|\varepsilon\ln\varepsilon|^{-1},&\mathrm{if}\;d=3,\\
\varepsilon^{-1},&\mathrm{if}\;d\geq4,
\end{cases}
\end{align*}
see \cite{AKLLL2007,BC1984,BLY2009,AKL2005,Y2007,Y2009,K1993} for $d=2$, \cite{BLY2009,LY2009,BLY2010,L2012} for $d=3$ , and \cite{BLY2009} for $d\geq4$, respectively. For more precise description regarding the singularities of the concentrated field, see \cite{KLY2013,ACKLY2013,KLY2014,LLY2019,LWX2019,BT2013}. The results have also been extended to the general $m$-convex perfect conductors, for example, see \cite{L2020,KLY2015,ZH2021}. For nonlinear equation, we refer to \cite{CS2019,CS201902,G2012}.

\subsection{Main results}
Before listing the main results of this paper, we first fix some notations and parameterize the domain. By picking a coordinate system appropriately, let $D_{1}$ and $D_{2}$ be translations of two touching insulators as follows:
\begin{align*}
D_{1}:=D_{1}^{\ast}+(0',\varepsilon/2),\quad\mathrm{and}\; D_{2}:=D_{2}^{\ast}+(0',-\varepsilon/2),
\end{align*}
where $D_{i}^{\ast}$, $i=1,2,$ satisfy
\begin{align*}
\partial D_{1}^{\ast}\cap\partial D_{2}^{\ast}=\{0\}\subset\mathbb{R}^{d},\quad\mathrm{and}\; D_{i}^{\ast}\subset\{(x',x_{d})\in\mathbb{R}^{d}\,|\,(-1)^{i+1}x_{d}>0\},\quad i=1,2.
\end{align*}
Here and afterwards, we will make use of superscript prime to represent $(d-1)$-dimensional variables and domains (for example, $x'$ and $B'$).

Suppose that for some $R_{0}>0$ independent of $\varepsilon$, let $\partial D_{1}$ and $\partial D_{2}$ around the origin be parameterized by two smooth curved surfaces $(x',\varepsilon/2+h_{1}(x'))$ and $(x',-\varepsilon/2+h_{2}(x'))$, respectively, and $h_{i}$, $i=1,2$ verify that for $m\geq2$ and $\gamma>0$,
\begin{enumerate}
{\it\item[(\bf{H1})]
$h_{1}(x')-h_{2}(x')=\lambda|x'|^{m}+O(|x'|^{m+\gamma}),\;\mathrm{if}\;x'\in B_{2R_{0}}',$
\item[(\bf{H2})]
$|\nabla_{x'}h_{i}(x')|\leq \kappa_{1}|x'|^{m-1},\;\mathrm{if}\;x'\in B_{2R_{0}}',$
\item[(\bf{H3})]
$\|h_{1}\|_{C^{2}(B'_{2R_{0}})}+\|h_{2}\|_{C^{2}(B'_{2R_{0}})}\leq \kappa_{2},$}
\end{enumerate}
where $\lambda$, $\kappa_{1}$ and $\kappa_{2}$ are three positive $\varepsilon$-independent constants. Here and in the following, the notation $O(A)$ implies that $|O(A)|\leq CA$ for some positive constant $C$ independent of $\varepsilon$. For $z'\in B'_{R_{0}},\,0<t\leq2R_{0}$, denote
\begin{align*}
\Omega_{t}(z'):=&\{x\in \mathbb{R}^{d}\,|\,-\varepsilon/2+h_{2}(x')<x_{d}<\varepsilon/2+h_{1}(x'),~|x'-z'|<{t}\}.
\end{align*}
We simplify the notation $\Omega_{t}(0')$ as $\Omega_{t}$ if $z'=0'$, and denote its top and bottom boundaries, respectively, by
\begin{align*}
\Gamma^{+}_{t}:=\{x\in\mathbb{R}^{d}\,|\,x_{d}=\varepsilon/2+h_{1}(x'),\;|x'|<t\},
\end{align*}
and
\begin{align*}
\Gamma^{-}_{t}:=\{x\in\mathbb{R}^{d}\,|\,x_{d}=-\varepsilon/2+h_{2}(x'),\;|x'|<t\}.
\end{align*}
Using the standard elliptic estimates, we get
\begin{align}\label{AZ001}
\|u\|_{C^{1}(\Omega\setminus\Omega_{R_{0}/2})}\leq C.
\end{align}
Then it is sufficient to quantify the singular behavior of $|\nabla u|$ in the small gap $\Omega_{R_{0}/2}$. That is, consider
\begin{align}\label{problem006}
\begin{cases}
\Delta u=0,&\hbox{in}\;\Omega_{2R_{0}},\\
\frac{\partial u}{\partial\nu}=0,&\mathrm{on}\;\Gamma^{\pm}_{2R_{0}},\\
\|u\|_{L^{\infty}(\Omega_{2R_{0}})}\leq1.
\end{cases}
\end{align}

Write
\begin{align}\label{degree}
\alpha=\alpha(d,m):=\frac{-(d+m-3)+\sqrt{(d+m-3)^{2}+4(d-2)}}{2}.
\end{align}
In fact, $\alpha(d,m)$ increases monotonically in terms of $d$, while it is monotonically decreasing in $m$. Moreover, we have
\begin{align*}
\alpha(d,m)=&
\begin{cases}
1-\frac{m}{d}+O(\frac{1}{d^{2}}),&\text{as }d\rightarrow\infty,\;\text{for any given }m\geq2,\\
\frac{d-2}{m}+O(\frac{1}{m^{2}}),&\text{as }m\rightarrow\infty,\;\text{for any given }d\geq3.
\end{cases}
\end{align*}

Unless otherwise stated, in the following we let $C$ be a constant which may differ at each occurrence, depending only on $d,m,\lambda,\gamma,R_{0},\kappa_{1},\kappa_{2}$, but not on $\varepsilon$. The first result is stated as follows.
\begin{theorem}\label{thm001}
Suppose that $D_{1},\,D_{2}\subset D\subseteq\mathbb{R}^{d}\,(d\geq3)$ are defined as above, conditions $\mathrm{(}${\bf{H1}}$\mathrm{)}$--$\mathrm{(}${\bf{H3}}$\mathrm{)}$ hold. Let $u\in H^{1}(\Omega_{2R_{0}})$ be the solution of \eqref{problem006}. Then for a sufficiently small $\varepsilon>0$ and $x\in\Omega_{R_{0}/2}$,
\begin{align}\label{U002}
|\nabla u(x)|\leq C\|u\|_{L^{\infty}(\Omega_{2R_{0}})}(\varepsilon+|x'|^{m})^{\frac{\alpha-1}{m}},
\end{align}
where $\alpha$ is given by \eqref{degree}.
\end{theorem}
\begin{remark}
Previously in \cite{LY2021,LY202102}, Li and Yang established the pointwise upper bound on the gradient as follows:
\begin{align}\label{DM001}
|\nabla u(x)|\leq C(\varepsilon+|x'|^{m})^{-1/m+\beta},\quad\mathrm{in}\;\Omega_{R_{0}/2},
\end{align}
for some inexplicit $\beta>0$. We further improve the upper bound in \eqref{DM001} by solving the explicit value of $\beta$, that is, $\beta=\frac{\alpha}{m}$, as shown in Theorem \ref{thm001}.
\end{remark}

\begin{remark}
It is worthwhile to emphasize that by using the change of variables \eqref{SCALING} below for every line segment in the narrow region $\Omega_{R_{0}/2}$ in the whole Section \ref{SEC02} which is different from the idea in \cite{DLY2021}, we avoid dividing into two cases to complete the proof of Theorem \ref{thm001} in the following and thus succeeds in simplifying the proof procedure in \cite{DLY2021}.

\end{remark}

\begin{remark}
The shape of inclusions considered in Theorem \ref{thm001} covers a class of axisymmetric inclusions as follows. To be precise, $\partial D_{1}$ and $\partial D_{2}$ are, respectively, expressed as
\begin{align}\label{ellipsoids}
|x'|^{m}+|x_{d}-\varepsilon/2-r_{1}|^{m}=r_{1}^{m},\quad\mathrm{and}\;|x'|^{m}+|x_{d}+\varepsilon/2+r_{2}|^{m}=r^{m}_{2},
\end{align}
where $r_{1}$ and $r_{2}$ are two positive $\varepsilon$-independent constants. Using Taylor expansion for \eqref{ellipsoids}, we have
\begin{align*}
h_{1}(x')-h_{2}(x')=\lambda_{0}|x'|^{m}+O(|x'|^{2m}),\quad\mathrm{in}\;\Omega_{r_{0}},
\end{align*}
where $\lambda_{0}=\frac{1}{m}\left(\frac{1}{r_{1}^{m-1}}+\frac{1}{r_{2}^{m-1}}\right)$ and $0<r_{0}<\min\{r_{1},r_{2}\}$.

\end{remark}

In order to prove the optimality of the blow-up rate $\varepsilon^{\frac{\alpha-1}{m}}$ obtained in Theorem \ref{thm001}, we now consider two ellipsoids in \eqref{ellipsoids} with $r_{1}=r_{2}=1$. The optimal lower bound on the gradient is established as follows.
\begin{theorem}\label{thm002}
For $d\geq3$, let $D:=B_{5}$ and $D_{i}$, $i=1,2$ be defined by \eqref{ellipsoids} with $r_{1}=r_{2}=1$. Let $u\in H^{1}(\Omega)$ be the solution of (\ref{con002}) with $\varphi=x_{1}$. Then for a sufficiently small $\varepsilon>0$,
\begin{align}\label{ARQZ001}
\|\nabla u\|_{L^{\infty}(\Omega\cap B_{2\sqrt[m]{\varepsilon}})}\geq\frac{1}{C}\varepsilon^{\frac{\alpha-1}{m}},
\end{align}
where $\alpha$ is given by \eqref{degree}.
\end{theorem}

In the following, we will establish the optimal pointwise upper and lower bounds on the gradient in Sections \ref{SEC02} and \ref{SEC03}, respectively.

\section{Pointwise upper bound on the gradient}\label{SEC02}
Without loss of generality, we let $\lambda=1$ in condition ({\bf{H1}}). Denote
\begin{align*}
\delta:=\delta(|y'|)=\varepsilon+|y'|^{m},\quad |y'|\leq 2R_{0}.
\end{align*}
For any fixed $x_{0}=(x_{0}',x_{d})\in\Omega_{R_{0}/2}$, write
\begin{align*}
\delta_{0}:=\delta(|x_{0}'|)=\varepsilon+|x_{0}'|^{m}.
\end{align*}
Define a cylinder as follows: for $s,t>0$ and $x'\in B'_{R_{0}}$,
\begin{align*}
Q_{s,t}(x'):=\{y=(y',y_{d})\in\mathbb{R}^{d}\,|\,|y'-x'|<s,\,|y_{d}|<t\}.
\end{align*}
For simplicity, let $Q_{s,t}:=Q_{s,t}(0')$ if $x'=0'$. Using a change of variables for $\Omega_{2R_{0}}$ as follows:
\begin{align}\label{SCALING} 
\begin{cases}
y'=x',\\
y_{d}=2\delta_{0}\left(\frac{x_{d}-h_{2}(x')+\varepsilon/2}{\varepsilon+h_{1}(x')-h_{2}(x')}-\frac{1}{2}\right),
\end{cases}
\end{align}
we derive a cylinder $Q_{2R_{0},\delta_{0}}$ of thickness $2\delta_{0}$. Write $v(y)=u(x)$. In light of equation \eqref{problem006}, we see that $v$ is a solution of
\begin{align}\label{ZKM001}
\begin{cases}
-\partial_{i}(a_{ij}(y)\partial_{j}v(y))=0,&\mathrm{in}\;Q_{2R_{0},\delta_{0}},\\
a_{dj}(y)\partial_{j}v(y)=0,&\mathrm{on}\;\{y_{d}=\pm\delta_{0}\},
\end{cases}
\end{align}
with $\|v\|_{L^{\infty}(Q_{R_{0},\delta_{0}})}\leq1$, where
\begin{gather}
\begin{align*}
(a_{ij}(y))=&\frac{2\delta_{0}(\partial_{x}y)(\partial_{x}y)^{t}}{\det(\partial_{x}y)}\notag\\
=&
\begin{pmatrix}\delta&0&\cdots&0&a_{1d} \\ 0&\delta&\cdots&0&a_{2d}\\ \vdots&\vdots&\ddots&\vdots&\vdots\\0&0&\cdots&\delta&a_{d-1\,1}\\ a_{d1}&a_{d2}&\cdots&a_{d\,d-1}&\frac{4\varepsilon^{2}+\sum^{d-1}_{i=1}a_{id}^{2}}{\delta}
\end{pmatrix}+\begin{pmatrix} e^{1}&0&\cdots&0 \\ 0&e^{2}&\cdots&0 \\ \vdots&\vdots&\ddots&\vdots\\
0&0&\cdots&e^{d}
\end{pmatrix},
\end{align*}
\end{gather}
whose elements satisfy that for $i=1,...,d-1,$ using conditions ({\bf{H1}}) and ({\bf{H2}}),
\begin{align}\label{K001}
|a_{id}|=|a_{di}|=|-2\delta_{0}\partial_{i}h_{2}(y')-(y_{d}+\delta_{0})\partial_{i}(h_{1}-h_{2})(y')|\leq C\delta_{0}|y'|^{m-1},
\end{align}
and
\begin{align}\label{K002}
|e^{i}|=|O(|y'|^{m+\gamma})|\leq C|y'|^{m+\gamma},\quad |e^{d}|\leq C\delta_{0}^{2}|y'|^{\gamma}\delta^{-1}.
\end{align}

In light of $\frac{\partial u}{\partial\nu}=0$ on $\Gamma^{\pm}_{2R_{0}}$, it follows from ({\bf{H2}}) and \eqref{DM001} that
\begin{align*}
|\partial_{d}u(x)|\leq C|x'|^{m-1}|\nabla_{x'}u|\leq C|x'|^{m-2},\quad\mathrm{on}\;\Gamma^{\pm}_{2R_{0}}.
\end{align*}
This, in combination with \eqref{AZ001}, the harmonicity of $\partial_{d}u$ and the maximum principle, yields that
\begin{align}\label{U001}
|\partial_{d}u|\leq C,\quad\mathrm{in}\;\Omega_{2R_{0}},
\end{align}
and then
\begin{align}\label{AZ005}
|\partial_{d} v|\leq C\delta_{0}^{-1}\delta,\quad\mathrm{in}\;Q_{2R_{0},\delta_{0}}.
\end{align}
Let
\begin{align}\label{KL01}
\bar{v}(y'):=\fint^{\delta_{0}}_{-\delta_{0}}v(y',y_{d})\,dy_{d}.
\end{align}
Then $\bar{v}$ verifies
\begin{align}\label{ZK003}
\mathrm{div}(\delta\nabla\bar{v})=\mathrm{div}F,\quad\mathrm{in}\;B'_{2R_{0}},
\end{align}
where $F=(F_{1},...,F_{d-1})$, $F_{i}:=-\overline{a_{id}\partial_{d}v}-e^{i}\partial_{i}\bar{v}$ for $i=1,...,d-1$, $\overline{a_{id}\partial_{d}v}$ represents the average of $a_{id}\partial_{d}v$ with respect to $y_{d}$ on $(-\delta_{0},\delta_{0})$. From \eqref{DM001} and \eqref{K001}--\eqref{AZ005}, we obtain
\begin{align}\label{QL001}
|F_{i}|\leq C\left(|y'|^{m-1}\delta+|y'|^{m+\gamma}\delta^{-1/m}\right),\;\, i=1,...,d-1,\quad\mathrm{in}\;B'_{2R_{0}}.
\end{align}
For $\gamma,\sigma\in\mathbb{R}$, define a norm as follows:
\begin{align}\label{QL002}
\|F\|_{\varepsilon,\gamma,\sigma,B_{R}'}:=\sup\limits_{y'\in B_{R}'}|y'|^{-\gamma}(\varepsilon+|y'|^{m})^{\sigma-1}|F(y')|,\;\,\mathrm{with}\;0<R\leq2R_{0}.
\end{align}

For any $0<R\leq 2R_{0}$, we decompose the solution $\bar{v}$ of \eqref{ZK003} as follows:
\begin{align}\label{ADAD001}
\bar{v}:=\bar{v}_{1}+\bar{v}_{2},\quad\mathrm{in}\;B_{R}',
\end{align}
where $\bar{v}_{i},i=1,2,$ satisfy
\begin{align}\label{de001}
\begin{cases}
\mathrm{div}(\delta\nabla\bar{v}_{1})=0,& \mathrm{in}\;B_{R}',\\
\bar{v}_{1}=\bar{v},&\mathrm{on}\;\partial B_{R}',
\end{cases}
\end{align}
and
\begin{align}\label{de002}
\begin{cases}
\mathrm{div}(\delta\nabla\bar{v}_{2})=\mathrm{div}F,&\mathrm{in}\;B_{R}',\\
\bar{v}_{2}=0,&\mathrm{on}\;\partial B_{R}',
\end{cases}
\end{align}
respectively.

For $\bar{v}_{1}$, we have
\begin{lemma}\label{lemma001}
For $d\geq 3$, let $\bar{v}_{1}$ be a solution of \eqref{de001}. Then for any $0<\rho<R$,
\begin{align*}
\left(\fint_{\partial B_{\rho}'}|\bar{v}_{1}-\bar{v}_{1}(0')|^{2}\right)^{\frac{1}{2}}\leq\left(\frac{\rho}{R}\right)^{\alpha}\left(\fint_{\partial B_{R}'}|\bar{v}_{1}-\bar{v}_{1}(0')|^{2}\right)^{\frac{1}{2}},
\end{align*}
where $\alpha$ is defined by \eqref{degree}.

\end{lemma}
\begin{proof}
Observe that $\bar{v}_{1}\in C^{\infty}(B_{R}')$ by using the standard elliptic theory. Without loss of generality, let $\bar{v}_{1}(0')=0$ and $R=1$. Write $y'=(r,\xi)\in(0,1)\times\mathbb{S}^{d-2}$. Then $\bar{v}_{1}$ verifies
\begin{align*}
\partial_{rr}\bar{v}_{1}+\left(\frac{d-2}{r}+\frac{mr^{m-1}}{\varepsilon+r^{m}}\right)\partial_{r}\bar{v}_{1}+\frac{1}{r^{2}}\Delta_{\mathbb{S}^{d-2}}\bar{v}_{1}=0,\quad\mathrm{in}\;B_{1}'\setminus\{0'\}.
\end{align*}
Adopt the following decomposition
\begin{align}\label{FK001}
\bar{v}_{1}(y')=\sum^{\infty}_{k=1}\sum^{N(k)}_{l=1}V_{k,l}(r)Y_{k,l}(\xi),\quad y'\in B_{1}'\setminus\{0'\},
\end{align}
where $\{Y_{k,l}\}_{k,l}$ is an orthonormal basis of $L^{2}(\mathbb{S}^{d-2})$ and every element $Y_{k,l}$ denotes a $k$-th degree spherical harmonics satisfying that
\begin{align*}
-\Delta_{\mathbb{S}^{d-2}}Y_{k,l}=k(k+d-3)Y_{k,l}.
\end{align*}
Therefore, $V_{k,l}(r)\in C^{2}(0,1)$ is determined by
\begin{align*}
V_{k,l}(r)=\int_{\mathbb{S}^{d-2}}\bar{v}_{1}(y')Y_{k,l}(\xi)d\xi,
\end{align*}
and verifies
\begin{align*}
L_{k}V_{k,l}:=\partial_{rr}V_{k,l}(r)+\left(\frac{d-2}{r}+\frac{mr^{m-1}}{\varepsilon+r^{m}}\right)\partial_{r}V_{k,l}(r)-\frac{k(k+d-3)}{r^{2}}V_{k,l}(r)=0,
\end{align*}
for $r\in (0,1)$, $k\in\mathbb{N}$, $l=1,2,...,N(k)$.

For every $c\in\mathbb{R}$, a direct calculation gives
\begin{align*}
L_{k}r^{c}=r^{c-2}\left[c^{2}+\Big(d-3+\frac{mr^{m}}{\varepsilon+r^{m}}\Big)c-k(k+d-3)\right].
\end{align*}
It then follows that for a sufficiently small $c>0$,
\begin{align}\label{W001}
L_{k}r^{-c}\leq0,\quad\mathrm{and}\; L_{k}r^{\alpha_{k}}\leq0,\quad\mathrm{in}\;(0,1),
\end{align}
where $\alpha_{k}$ is given by
\begin{align*}
\alpha_{k}:=\frac{-(d+m-3)+\sqrt{(d+m-3)^{2}+4k(k+d-3)}}{2},\quad\text{for}\;k\in\mathbb{N}.
\end{align*}
Especially when $k=1$, we have $\alpha_{1}=\alpha$. Then we obtain from \eqref{W001} that for every $\tau>0$,
\begin{align*}
L_{k}(\pm V_{k,l}(r)-\tau r^{-c}-|V_{k,l}(1)|r^{\alpha_{k}})\geq0,\quad\mathrm{in}\;(0,1).
\end{align*}
Note that $V_{k,l}$ remains bounded in $(0,1)$ for the sake of $\bar{v}_{1}\in L^{\infty}(B_{1})$. Hence,
\begin{align*}
\pm V_{k,l}(r)-\tau r^{-c}-|V_{k,l}(1)|r^{\alpha_{k}}<0,\quad\mathrm{as}\;r\searrow0,\;\mathrm{or}\;r=1,
\end{align*}
which, together with the maximum principle, yields that
\begin{align*}
|V_{k,l}(r)|\leq\tau r^{-c}+|V_{k,l}(1)|r^{\alpha_{k}},\quad\mathrm{in}\;(0,1).
\end{align*}
Letting $\tau\rightarrow0$, we get
\begin{align*}
|V_{k,l}(r)|\leq|V_{k,l}(1)|r^{\alpha_{k}},\quad\mathrm{for}\;r\in(0,1).
\end{align*}
This, in combination with \eqref{FK001}, reads that
\begin{align*}
\fint_{\partial B_{\rho}}|\bar{v}_{1}|^{2}=&\sum^{\infty}_{k=1}\sum^{N(k)}_{l=1}|V_{k,l}(\rho)|^{2}\leq\rho^{2\alpha}\sum^{\infty}_{k=1}\sum^{N(k)}_{l=1}|V_{k,l}(1)|^{2}=\rho^{2\alpha}\fint_{\partial B_{1}}|\bar{v}_{1}|^{2}.
\end{align*}

\end{proof}
%Let
%\begin{align}
%\begin{cases}
%\gamma-m\sigma>\frac{\sqrt{m^{2}+8m-16}-m}{4},&\mathrm{if}\;d=3,\\
%\gamma-m\sigma>\frac{\sqrt{m^{2}+8m-16}-m}{4}\;\,\mathrm{or}\;\gamma-m\sigma<-\frac{(m-2)(d-1)}{2(d-3)},&\mathrm{if}\;d>3.
%\end{cases}
%\end{align}

With regard to $\bar{v}_{2}$, we make use of Moser's iteration argument to obtain the following result.
\begin{lemma}\label{lemma002}
For $d\geq 3$, $1+\gamma-m\sigma>0$, let $\bar{v}_{2}$ be a solution of \eqref{de002} with $R=1$. Suppose that $F\in L^{\infty}(B_{1}')$ satisfies $\|F\|_{\varepsilon,\gamma,\sigma,B_{1}'}<\infty$.  Then,
\begin{align*}
\|\bar{v}_{2}\|_{L^{\infty}(B_{1}')}\leq C\|F\|_{\varepsilon,\gamma,\sigma,B_{1}'},
\end{align*}
where $C$ is a positive constant depending only on $d,m,\gamma,\sigma$, but not on $\varepsilon$.
\end{lemma}
\begin{proof}
For simplicity, let $\|F\|_{\varepsilon,\gamma,\sigma,B_{1}'}=1$. Write $r=|y'|$. In view of \eqref{de002}, we see that $\bar{v}_{2}$ satisfies
\begin{align}\label{AWA001}
\Delta \bar{v}_{2}+mr^{m-1}\delta^{-1}\partial_{r}\bar{v}_{2}=\partial_{i}(F_{i}\delta^{-1})+mF_{i}y_{i}|y'|^{m-2}\delta^{-2},\quad\mathrm{in}\;B_{1}'.
\end{align}
From \eqref{QL002}, we have
\begin{align*}
|F_{i}\delta^{-1}|\leq r^{\gamma-m\sigma},\quad |F_{i}y_{i}|y'|^{m-2}\delta^{-2}|\leq r^{\gamma-m\sigma-1}.
\end{align*}
Multiplying equation \eqref{AWA001} by $-|\bar{v}_{2}|^{p-2}\bar{v}_{2}$ with $p\geq2$, it follows from integration by parts that
\begin{align*}
&(p-1)\int_{B_{1}'}|\nabla\bar{v}_{2}|^{2}|\bar{v}_{2}|^{p-2}\notag\\
&\leq C(p-1)\int_{B'_{1}}|\nabla\bar{v}_{2}||\bar{v}_{2}|^{p-2}r^{\gamma-m\sigma}+C\int_{B_{1}'}|\bar{v}_{2}|^{p-1}r^{\gamma-m\sigma-1},
\end{align*}
where we utilized the fact that
\begin{align*}
\int_{B'_{1}}mr^{m-1}\delta^{-1}\partial_{r}\bar{v}_{2}(|\bar{v}_{2}|^{p-2}\bar{v}_{2})=&\frac{1}{p}\int_{\mathbb{S}^{d-2}}\int^{1}_{0}mr^{d+m-3}\delta^{-1}\partial_{r}|\bar{v}_{2}|^{p}drd\theta\notag\\
=&-\frac{1}{p}\int_{\mathbb{S}^{d-2}}\int_{0}^{1}\partial_{r}(mr^{d+m-3}\delta^{-1})|\bar{v}_{2}|^{p}drd\theta\leq0.
\end{align*}
Then in view of $1+\gamma-m\sigma>0$, it follows from H\"{o}lder's inequality that
\begin{align*}
&(p-1)\int_{B_{1}'}|\nabla\bar{v}_{2}|^{2}|\bar{v}_{2}|^{p-2}\notag\\
&\leq C(p-1)\||\nabla\bar{v}_{2}||\bar{v}_{2}|^{\frac{p-2}{2}}\|_{L^{2}(B_{1}')}\|\bar{v}_{2}^{p-2}\|^{1/2}_{L^{\frac{d-1+2\tau}{d-3+2\tau}}(B_{1}')}\|r^{2(\gamma-m\sigma)}\|_{L^{\frac{d-1}{2}+\tau}(B'_{1})}^{1/2}\notag\\
&\quad\,+C\|\bar{v}_{2}^{p-1}\|_{L^{\frac{d-1+2\tau}{d-3+2\tau}}(B_{1}')}\|r^{\gamma-m\sigma-1}\|_{L^{\frac{d-1}{2}+\tau}(B_{1}')}\notag\\
&\leq C(p-1)\||\nabla\bar{v}_{2}||\bar{v}_{2}|^{\frac{p-2}{2}}\|_{L^{2}(B_{1}')}\|\bar{v}_{2}^{p-2}\|^{1/2}_{L^{\frac{d-1+2\tau}{d-3+2\tau}}(B_{1}')}+C\|\bar{v}_{2}^{p-1}\|_{L^{\frac{d-1+2\tau}{d-3+2\tau}}(B_{1}')},
\end{align*}
where $\tau>0$ is a sufficiently small constant such that
\begin{align*}
&\|r^{2(\gamma-m\sigma)}\|_{L^{\frac{d-1}{2}+\tau}(B'_{1})}+\|r^{\gamma-m\sigma-1}\|_{L^{\frac{d-1}{2}+\tau}(B_{1}')}\leq C.
\end{align*}
Then using H\"{o}lder's inequality and Young's inequality, we get
\begin{align}\label{QAZ001}
\frac{2(p-1)}{p^{2}}\int_{B'_{1}}\big|\nabla|\bar{v}_{2}|^{\frac{p}{2}}\big|^{2}=&\frac{p-1}{2}\int_{B_{1}'}|\nabla\bar{v}_{2}|^{2}|\bar{v}_{2}|^{p-2}\notag\\
\leq&\max_{1\leq i\leq2}Cp^{i-1}\|\bar{v}_{2}\|^{p-i}_{L^{\frac{(d-1+2\tau)p}{d-3+2\tau}}(B_{1}')}.
\end{align}
In particular, if we pick $p=2$ in \eqref{QAZ001}, then we deduce from the Sobolev-Poincar\'{e} inequality and Young's inequality that
\begin{align}\label{E001}
\|\bar{v}_{2}\|_{L^{\frac{2(d-1+2\tau)}{d-3+2\tau}}(B_{1}')}\leq C.
\end{align}
Utilizing the Sobolev-Poincar\'{e} inequality and Young's inequality again for \eqref{QAZ001} with $p\geq2$, we obtain
\begin{align}\label{E002}
\|\bar{v}_{2}\|_{L^{tp}(B_{1}')}\leq&\max_{1\leq i\leq2}(Cp^{i})^{1/p}\left(\frac{p-i}{p}\|\bar{v}_{2}\|_{L^{\frac{(d-1+2\tau)p}{d-3+2\tau}}(B_{1}')}+\frac{i}{p}\right)\notag\\
\leq& (Cp^{2})^{1/p}\left(\|\bar{v}_{2}\|_{L^{\frac{(d-1+2\tau)p}{d-3+2\tau}}(B_{1}')}+\frac{2}{p}\right),
\end{align}
where $t:=t(d)$ is given by
\begin{align*}
\begin{cases}
t>\frac{d-1+2\tau}{d-3+2\tau},&d=3,\\
t=\frac{d-1}{d-3},&d>3.
\end{cases}
\end{align*}
Set
\begin{align*}
p_{k}=2\left(\frac{(d-3+2\tau)t}{d-1+2\tau}\right)^{k}\frac{d-1+2\tau}{d-3+2\tau},\quad k\geq0,\,d\geq3.
\end{align*}
Hence by iteration with \eqref{E001}--\eqref{E002}, we get
\begin{align*}
\|\bar{v}_{2}\|_{L^{p_{k}}(B_{1}')}\leq&\prod^{k-1}_{i=0}(Cp_{i}^{2})^{1/p_{i}}\|\bar{v}_{2}\|_{L^{p_{0}}(B_{1}')}+\sum^{k-1}_{i=0}\prod^{k-1-i}_{j=0}(Cp_{k-1-j}^{2})^{1/p_{k-1-j}}\frac{2}{p_{i}}\notag\\
\leq&C\|\bar{v}_{2}\|_{L^{\frac{2(d-1+2\tau)}{d-3+2\tau}}(B_{1}')}+C\leq C,
\end{align*}
where $C=C(d,m,\gamma,\sigma)$ is independent of $k$. Then Lemma \ref{lemma002} is proved by letting $k\rightarrow\infty$.

\end{proof}

Combining Lemma \ref{lemma001} and \ref{lemma002}, we obtain
\begin{prop}\label{prop001}
For $d\geq 3$, $\sigma\geq0$, $1+\gamma-m\sigma>0$, $1+\gamma-m\sigma\neq\alpha$, let $\bar{v}$ be a solution of \eqref{ZK003} with $\|F\|_{\varepsilon,\gamma,\sigma,B_{R_{0}}'}<\infty$.  Then for any $R\in(0,R_{0})$,
\begin{align*}
\left(\fint_{\partial B_{R}'}|\bar{v}-\bar{v}(0')|^{2}\right)^{1/2}\leq C\|F\|_{\varepsilon,\gamma,\sigma,B_{R_{0}}'}R^{\tilde{\alpha}},
\end{align*}
where $\tilde{\alpha}:=\min\{\alpha,1+\gamma-m\sigma\}$ with $\alpha$ given by \eqref{degree}.

\end{prop}

\begin{proof}
Without loss of generality, set $\bar{v}(0')=0$ and $\|F\|_{\varepsilon,\gamma,\sigma,B_{R_{0}}'}=1$. For $0<\rho\leq R\leq R_{0}$, denote
\begin{align*}
\omega(\rho):=\bigg(\fint_{\partial B_{\rho}'}|\bar{v}|^{2}\bigg)^{\frac{1}{2}}.
\end{align*}
Write $\tilde{v}_{2}(y'):=\bar{v}_{2}(Ry')$. Then in light of \eqref{de002}, we see that $\tilde{v}_{2}$ verifies
\begin{align*}
\mathrm{div}\big((R^{-m}\varepsilon+|y'|^{m})\nabla\tilde{v}_{2}\big)=\mathrm{div}\tilde{F},\quad\mathrm{in}\;B_{1}',
\end{align*}
where $\tilde{F}(y'):=R^{-(m-1)}F(Ry')$ satisfies
\begin{align*}
\|\tilde{F}\|_{R^{-m}\varepsilon,\gamma,\sigma,B_{1}'}=R^{1+\gamma-m\sigma}\|F\|_{\varepsilon,\gamma,\sigma,B_{R}'}.
\end{align*}
Then using Lemma \ref{lemma002} for $\tilde{v}_{2}$ with $R^{-m}\varepsilon$ substituting for $\varepsilon$, we get
\begin{align}\label{FNM001}
\|\bar{v}_{2}\|_{L^{\infty}(B_{R}')}\leq CR^{1+\gamma-m\sigma}.
\end{align}
Recalling decomposition \eqref{ADAD001}, it follows from Lemma \ref{lemma001} and \eqref{FNM001} that
\begin{align}\label{GAZ001}
\omega(\rho)\leq&\bigg(\fint_{\partial B_{\rho}'}|\bar{v}_{1}-\bar{v}_{1}(0')|^{2}\bigg)^{\frac{1}{2}}+\bigg(\fint_{\partial B_{\rho}'}|\bar{v}_{2}-\bar{v}_{2}(0')|^{2}\bigg)^{\frac{1}{2}}\notag\\
\leq&\left(\frac{\rho}{R}\right)^{\alpha}\bigg(\fint_{\partial B_{R}'}|\bar{v}_{1}|^{2}\bigg)^{\frac{1}{2}}+\left(\frac{\rho}{R}\right)^{\alpha}|\bar{v}_{1}(0')|+2\|\bar{v}_{2}\|_{L^{\infty}(B_{R}')}\notag\\
\leq&\left(\frac{\rho}{R}\right)^{\alpha}\omega(R)+CR^{1+\gamma-m\sigma},
\end{align}
where we also used the fact that $\bar{v}=\bar{v}_{1}$ on $\partial B_{R}'$ and $|\bar{v}_{1}(0')|=|\bar{v}_{2}(0')|$ in virtue of $\bar{v}(0')=\bar{v}_{1}(0')+\bar{v}_{2}(0')=0$. For $i=0,...,k-1,$ $k$ is a positive integer, let $\rho=2^{-i-1}R_{0}$ and $R=2^{-i}R_{0}$ in \eqref{GAZ001}. Since $1+\gamma-m\sigma\neq\alpha$, it then follows from $k$ iterations that
\begin{align*}
\omega(2^{-k}R_{0})\leq&2^{-k\alpha}\omega(R_{0})+C\sum^{k}_{i=1}2^{-(k-i)\alpha}(2^{1-i}R_{0})^{1+\gamma-m\sigma}\notag\\
\leq&2^{-k\alpha}\omega(R_{0})+C2^{-k\alpha}R_{0}^{1+\gamma-m\sigma}\frac{1-2^{k(\alpha-1-\gamma+m\sigma)}}{1-2^{\alpha-1-\gamma+m\sigma}}\notag\\
\leq&2^{-k\tilde{\alpha}}\big(\omega(R_{0})+CR_{0}^{1+\gamma-m\sigma}\big).
\end{align*}
For every $\rho\in(0,R_{0})$, there exists some integer $k$ such that $\rho\in(2^{-k-1}R_{0},2^{-k}R_{0}]$. Hence we have
\begin{align*}
\omega(\rho)\leq C\rho^{\tilde{\alpha}},\quad\mathrm{for}\;\mathrm{any}\;\rho\in(0,R_{0}).
\end{align*}
The proof is complete.

\end{proof}

We are now ready to prove Theorem \ref{thm001}.
\begin{proof}[Proof of Theorem \ref{thm001}]
To begin with, we point out again that by making use of the change of variables in \eqref{SCALING} for every line segment in the thin gap $\Omega_{R_{0}/2}$, we can achieve a united proof of Theorem \ref{thm001}. That is, we don't need to divide into two cases to prove Theorem \ref{thm001} any more, which simplifies the corresponding proof procedure in \cite{DLY2021}.

Suppose that $\lambda=1$, $u(0)=0$ and $\|u\|_{L^{\infty}(\Omega_{R_{0}})}=1$ without loss of generality. Let $v$ and $\bar{v}$ be defined by \eqref{ZKM001} and \eqref{KL01}--\eqref{ZK003}, respectively. From \eqref{QL001}, we obtain
\begin{align*}
\|F\|_{\varepsilon,\gamma,\sigma_{0},B_{R_{0}}'}<\infty,\quad\mathrm{with}\;\sigma_{0}=\frac{1}{m},
\end{align*}
decreasing $\gamma$ if necessary. Using \eqref{AZ005}, we have
\begin{align}\label{LZMWN001A}
|v(y',y_{d})-\bar{v}(y')|\leq2\delta_{0}\max_{y_{d}\in(-\delta_{0},\delta_{0})}|\partial_{d}v(y',y_{d})|\leq C\delta,\quad\mathrm{in}\;Q_{R_{0},\delta_{0}}.
\end{align}
From Proposition \ref{prop001}, we have
\begin{align*}
\int_{B'_{2c_{0}\delta^{1/m}_{0}}(x_{0}')}|\bar{v}-\bar{v}(0')|^{2}\leq \int_{B'_{|x_{0}'|+2c_{0}\delta^{1/m}_{0}}(0')}|\bar{v}-\bar{v}(0')|^{2}\leq C\delta^{\frac{2\tilde{\alpha}+n-1}{m}}_{0}.
\end{align*}
This, in combination with \eqref{LZMWN001A}, yields that
\begin{align*}
&\fint_{Q_{2c_{0}\delta_{0}^{1/m},\delta_{0}}(x_{0}')}|v-\bar{v}(0')|^{2}dy\notag\\
&\leq\fint_{Q_{2c_{0}\delta_{0}^{1/m},\delta_{0}}(x_{0}')}2\big(|v-\bar{v}|^{2}+|\bar{v}-\bar{v}(0')|^{2})dy\leq C\delta_{0}^{\frac{2\tilde{\alpha}}{m}},
\end{align*}
where $c_{0}:=2^{-(m+1)}m^{-1}$. Let
\begin{align*}
\tilde{v}(y)=&v(\delta_{0}^{1/m}y'+x_{0}',\delta_{0}^{1/m}y_{d})-\bar{v}(0'),\notag\\
\tilde{a}_{ij}(y)=&\delta_{0}^{-1}a_{ij}(\delta_{0}^{1/m}y'+x_{0}',\delta_{0}^{1/m}y_{d}).
\end{align*}
Since $Q_{2c_{0}\delta_{0}^{1/m},\delta_{0}}(x_{0}')\subset Q_{2R_{0},\delta_{0}}$ for $x_{0}'\in B_{R_{0}/2}'$, then $\tilde{v}$ solves
\begin{align*}
\begin{cases}
-\partial_{i}(\tilde{a}_{ij}(y)\partial_{j}\tilde{v}(y))=0,&\mathrm{in}\; Q_{2c_{0},\delta_{0}^{1-1/m}},\\
\tilde{a}_{dj}(y)\partial_{j}\tilde{v}(y)=0,&\mathrm{on}\;\{y_{d}=\pm\delta_{0}^{1-1/m}\}.
\end{cases}
\end{align*}
Observe that for $x=(x',x_{d})\in\Omega_{s}(x_{0}')$, $0<s\leq 2c_{0}\delta_{0}^{1/m}$, $c_{0}=2^{-(m+1)}m^{-1}$, we deduce
\begin{align*}
|\delta(x')-\delta(x_{0}')|=&||x'|^{m}-|x_{0}'|^{m}|\leq m|x'_{\theta}|^{m-1}|x'-x_{0}'|\notag\\
\leq&2^{m-2}ms(s^{m-1}+|x_{0}'|^{m-1})\leq\frac{\delta(x_{0}')}{2},
\end{align*}
where $x_{\theta}'$ is some point between $x_{0}'$ and $x'$. Then, we have
\begin{align}\label{QWN001}
\frac{1}{2}\delta(x_{0}')\leq\delta(x')\leq\frac{3}{2}\delta(x_{0}'),\quad\mathrm{in}\;\Omega_{s}(x_{0}').
\end{align}
Remark that the result in \eqref{QWN001} actually gives a precise characterization for the equivalence of the length of each line segment in the small narrow region $\Omega_{s}(x_{0}')$, which is not presented in previous work \cite{DLY2021}. Using \eqref{QWN001}, we obtain that the coefficient matrix $\tilde{a}:=(\tilde{a}_{ij})$ satisfies
\begin{align*}
\frac{I}{C}\leq\tilde{a}\leq CI,\quad\mathrm{and}\;\|\tilde{a}\|_{C^{\mu}(Q_{2c_{0},\delta_{0}^{1-1/m}})}\leq C,\quad\mathrm{for}\;\mathrm{any}\;\mu\in(0,1].
\end{align*}
For any integer $l$, denote
\begin{align*}
S_{l}:=\{y\in\mathbb{R}^{d}\,|\,|y'|<2c_{0},\;(2l-1)\delta_{0}^{1-1/m}<y_{d}<(2l+1)\delta_{0}^{1-1/m}\},
\end{align*}
and
\begin{align*}
S:=\{y\in\mathbb{R}^{d}\,|\,|y'|<2c_{0},\,|y_{d}|<2c_{0}\}.
\end{align*}
In particular, $S_{0}=Q_{2c_{0},\delta_{0}^{1-1/m}}$. Introduce a new function as follows:
\begin{align*}
\hat{v}(y):=\tilde{v}(y',(-1)^{l}(y_{d}-2l\delta_{0}^{1-1/m})),\quad\mathrm{in}\;S_{l},\,l\in\mathbb{Z},
\end{align*}
which is generated by performing even extension of $\tilde{v}$ with respect to $y_{d}=\delta_{0}^{1-1/m}$ and then the periodic extension with the period $4\delta_{0}^{1-1/m}$. The corresponding coefficients become that for $k=1,...,d-1$ and any $l\in\mathbb{Z}$,
\begin{align*}
\hat{a}_{dk}(y)=\hat{a}_{kd}(y):=(-1)^{l}\tilde{a}_{kd}(y',(-1)^{l}(y_{d}-2l\delta_{0}^{1-1/m})),\quad\mathrm{in}\;S_{l},
\end{align*}
and
\begin{align*}
\hat{a}_{ij}(y):=\tilde{a}_{ij}(y',(-1)^{l}(y_{d}-2l\delta_{0}^{1-1/m})),\quad\mathrm{in}\;S_{l},
\end{align*}
for other indices. Therefore, $\hat{v}$ and $\hat{a}_{ij}$ are defined in $Q_{2,\infty}$. From the conormal boundary conditions, we know that $\hat{v}$ verifies
\begin{align*}
\partial_{i}(\hat{a}_{ij}\partial_{j}\hat{v})=0,\quad\mathrm{in}\;S.
\end{align*}
Applying Proposition 4.1 of \cite{LN2003} and Lemma 2.1 of \cite{LY202102}, we get
\begin{align*}
\|\nabla\hat{v}\|_{L^{\infty}(\frac{1}{2}S)}\leq C\|\hat{v}\|_{L^{2}(S)}\leq C\delta_{0}^{\frac{\tilde{\alpha}}{m}}.
\end{align*}
Then back to $u$, we obtain that for $x_{0}=(x_{0}',x_{d})\in\Omega_{R_{0}/2}$,
\begin{align*}
|\nabla u(x_{0})|\leq\|\nabla u\|_{L^{\infty}(\Omega_{c_{0}\delta^{1/m}}(x_{0}'))}\leq C\delta_{0}^{\frac{\tilde{\alpha}-1}{m}}=C(\varepsilon+|x_{0}'|^{m})^{\frac{\tilde{\alpha}-1}{m}}.
\end{align*}
Then we improve the previous upper bound $|\nabla u(x)|\leq C(\varepsilon+|x'|^{m})^{-\sigma_{0}}$ to be $|\nabla u(x)|\leq C(\varepsilon+|x'|^{m})^{\frac{\tilde{\alpha}-1}{m}},$ where $\frac{\tilde{\alpha}-1}{m}=\min\{\frac{\alpha-1}{m},-\sigma_{0}+\frac{\gamma}{m}\}.$ If $1+\gamma-m\sigma_{0}>\alpha$, then the proof is finished. Otherwise, if $1+\gamma-m\sigma_{0}<\alpha$, then pick $\sigma_{1}=\sigma_{0}-\frac{\gamma}{m}$ and repeat the above argument. It may need to decrease $\gamma$ if necessary such that $\frac{\alpha-1}{m}\neq-\sigma_{0}+k\frac{\gamma}{m}$ for any $k\geq1$. By repeatedly using the argument for finite times, we complete the proof of Theorem \ref{thm001}.

\end{proof}

\section{optimal lower bound on the gradient}\label{SEC03}

Denote
\begin{align}\label{lam}
\lambda_{0}:=\frac{2}{m}.
\end{align}
We start by proving the following lemma for the purpose of establishing the optimal lower bound on the gradient.
\begin{lemma}\label{lemma006}
For $\varepsilon>0$, there exists a unique solution $g\in L^{\infty}((0,1))\cap C^{\infty}((0,1])$
\begin{align}\label{AMR001}
Lg:=\partial_{rr}g(r)+\left(\frac{d-2}{r}+\frac{m\lambda_{0}r^{m-1}}{\varepsilon+\lambda_{0}r^{m}}\right)\partial_{r}g(r)-\frac{d-2}{r^{2}}g(r)=0,\quad0<r<1,
\end{align}
such that $g(1)=1$. Furthermore, $g\in C([0,1])$ increases strictly with $g(0)=0$, satisfying that for $\beta\geq\frac{2\alpha^{2}+\alpha(d+m-3)}{2\alpha+d-3}$,
\begin{align}\label{DAM001}
\min\{r,\lambda_{0}^{\frac{\beta-\alpha}{m}}r^{\beta}(\varepsilon+\lambda_{0}r^{m})^{\frac{\alpha-\beta}{m}}\}<g(r)<r^{\alpha},\quad\mathrm{in}\;(0,1),
\end{align}
and
\begin{align}\label{DAM002}
g(r)<C_{0}(\varepsilon)r,\quad\mathrm{in}\;(0,r_{0}(\varepsilon)),
\end{align}
where $\alpha$ is given by \eqref{degree}, $\lambda_{0}$ is defined in \eqref{lam}, and
\begin{align}\label{WZQ001}
r_{0}(\varepsilon)=\left(\frac{a_{0}(b_{0}-1)(d+b_{0}-2)}{m\lambda_{0}(1+a_{0}b_{0})}\varepsilon\right)^{\frac{1}{m+1-b_{0}}},\quad C_{0}(\varepsilon)=\frac{(r_{0}(\varepsilon))^{\alpha-1}}{1-a_{0}(r_{0}(\varepsilon))^{b_{0}-1}},
\end{align}
for any fixed constants $a_{0}>0$ and $1<b_{0}<m+1$.

\end{lemma}
\begin{remark}
We improve the corresponding estimates in Lemma 3.1 of \cite{DLY2021} by solving the explicit constants in \eqref{WZQ001}.
\end{remark}

\begin{proof}
Denote by $g_{\tau}\in C^{2}([\tau,1])$ the solution of $Lg_{\tau}=0$ in $(\tau,1)$ with $g_{\tau}(\tau)=\tau$ and $g_{\tau}(1)=1$, where $0<\tau<1$. Due to the fact that $Lr>0$ and $Lr^{\alpha}<0$ in $(0,1)$, it follows from the maximum principle and strong maximum principle that
\begin{align*}
r<g_{\tau}(r)<r^{\alpha},\quad r\in(\tau,1).
\end{align*}
Then there exists a solution $g\in C([0,1])\cap C^{\infty}((0,1])$ of $Lg=0$ in $(0,1)$ such that $g_{\tau}\rightarrow g$ in $C^{2}_{loc}((0,1])$ as $\tau\rightarrow0$ along a subsequence. Moreover, $r\leq g(r)\leq r^{\alpha}$ in $(0,1)$ and $h(0)=0$. Using the strong maximum principle, we further get $r<g(r)<r^{\alpha}$ in $(0,1)$.

Consider $\underline{g}(r):=\lambda_{0}^{\frac{\beta-\alpha}{m}}r^{\beta}(\varepsilon+\lambda_{0}r^{m})^{\frac{\alpha-\beta}{m}}$ for $\beta\in\mathbb{R}$. Then it follows from a direct calculation that
\begin{align*}
L\underline{g}=&\lambda_{0}^{\frac{\beta-\alpha}{m}}r^{\beta-2}(\varepsilon+\lambda_{0}r^{m})^{\frac{\alpha-\beta}{m}}\bigg((\beta-\alpha)^{2}\Big(\frac{\lambda_{0}r^{m}}{\varepsilon+\lambda_{0}r^{m}}\Big)^{2}\notag\\ &+\big((\alpha-\beta)(d+m+2\beta-3)+m\beta\big)\Big(\frac{\lambda_{0}r^{m}}{\varepsilon+\lambda_{0}r^{m}}\Big)+(d-2+\beta)(\beta-1)\bigg),\;\,\mathrm{in}\;(0,1).
\end{align*}
Denote
\begin{align*}
p(t):=(\beta-\alpha)^{2}t^{2}+\big((\alpha-\beta)(d+m+2\beta-3)+m\beta\big)t+(d-2+\beta)(\beta-1),
\end{align*}
for $0\leq t\leq1$. In light of $p(1)=0$, it suffices to require that
\begin{align*}
p'(t)\leq&2(\beta-\alpha)^{2}+(\alpha-\beta)(d+m+2\beta-3)+m\beta\notag\\
\leq&-(2\alpha+d-3)\beta+2\alpha^{2}+\alpha(d+m-3)\leq0,
\end{align*}
for the purpose of $L\underline{g}\geq0$ in $(0,1)$. This, together with the strong maximum principle and the fact that $\underline{g}(0)=g(0)$ and $\underline{g}(1)<g(1)$, yields that $\underline{g}$ is a subsolution of \eqref{AMR001} in the case of $\beta\geq\frac{2\alpha^{2}+\alpha(d+m-3)}{2\alpha+d-3}$.

On the other hand, let $\overline{g}:=C_{0}(\varepsilon)(r-a_{0}r^{b_{0}})$, where $a_{0}>0$ and $1<b_{0}<m+1$, and $C_{0}(\varepsilon)$ is given in \eqref{WZQ001}. A straightforward computation shows that for $r\in(0,r_{0}(\varepsilon))$,
\begin{align*}
L\overline{g}=C_{0}(\varepsilon)\left(-a_{0}(b_{0}-1)(b_{0}+d-2)r^{b_{0}-2}+\frac{m\lambda_{0}r^{m-1}}{\varepsilon+\lambda_{0}r^{m}}(1-a_{0}b_{0}r^{b_{0}-1})\right)\leq0,
\end{align*}
where $r_{0}(\varepsilon)$ and $C_{0}(\varepsilon)$ are given by \eqref{WZQ001}. Since $\overline{g}(0)=g(0)$ and $\overline{g}(r_{0}(\varepsilon))>g(r_{0}(\varepsilon))$, then we deduce from the strong maximum principle that $g(r)<C_{0}(\varepsilon)r$ in $(0,r_{0}(\varepsilon))$.

In addition, $g$ is actually strictly increasing in $(0,1)$. Otherwise, there exists a constant $r^{\ast}\in(0,1)$ such that $g'(r^{\ast})=0$ and $g''(r^{\ast})\leq0$. Then we get $Lg(r^{\ast})<0$ in virtue of $g(r^{\ast})>0$. This leads to a contradiction. It remains to prove the uniqueness of $g$. Assume that there exists another solution $g_{1}\in L^{\infty}((0,1))\cap C^{\infty}((0,1])$ of \eqref{AMR001} such that $g_{1}(1)=1$. Let $w:=g_{1}g^{-1}$ in $(0,1)$. Then
\begin{align*}
Lg_{1}=\frac{g}{G}(Gw')'=0,\quad\mathrm{in}\;(0,1),
\end{align*}
where $G=g^{2}r^{d-2}(\varepsilon+\lambda_{0}r^{m})$. Hence there are two constants $C_{1}$ and $C_{2}$ such that
\begin{align*}
g_{1}(r)=&g(r)\int_{r}^{1}\frac{C_{1}}{g^{2}(s)s^{d-2}(\varepsilon+\lambda_{0}s^{m})}\,ds+C_{2} g(r),\quad\mathrm{in}\;(0,1).
\end{align*}
From \eqref{DAM001}--\eqref{DAM002}, we obtain
\begin{align*}
&g(r)\int_{r}^{1}\frac{1}{g^{2}(s)s^{d-2}(\varepsilon+\lambda_{0}s^{m})}\,ds\notag\\
&\geq g(r)\int_{r}^{r_{0}(\varepsilon)}\frac{1}{g^{2}(s)s^{d-2}(\varepsilon+\lambda_{0}s^{m})}\,ds\notag\\
&\geq\frac{1}{(d-1)(C_{0}(\varepsilon))^{2}(\varepsilon+\lambda_{0}(r_{0}(\varepsilon))^{m})}\left(r^{2-d}-r(r_{0}(\varepsilon))^{1-d}\right)\rightarrow\infty,\quad\mathrm{as}\;r\rightarrow0,
\end{align*}
which, in combination with the fact that $g$ and $g_{1}$ are bounded, leads to that $C_{1}=0$, $C_{2}=1$ and thus $g=g_{1}$.

\end{proof}

\begin{proof}[Proof of Theorem \ref{thm002}]
\noindent{\bf Step 1.} To begin with, it follows from Taylor expansion that
\begin{align*}
h_{1}(x')=-h_{2}(x')=\frac{|x'|^{m}}{m}+O(|x'|^{2m}),\quad\mathrm{in}\;B_{1}'.
\end{align*}
Let
\begin{align*}
\bar{u}(x')=\fint_{-\frac{\varepsilon}{2}+h_{2}}^{\frac{\varepsilon}{2}+h_{1}}u(x',x_{d})\,dx_{d},\quad\mathrm{in}\;B_{1}'.
\end{align*}
Hence $\bar{u}$ is a solution of
\begin{align*}
\mathrm{div}((\varepsilon+\lambda_{0}|x'|^{m})\nabla\bar{u})=\mathrm{div}F,\quad \mathrm{in}\;B_{1}',
\end{align*}
where $\lambda_{0}$ is given in \eqref{lam}, $F=(F_{1},...,F_{d-1})$, $F_{i}=2|x'|^{m-2}(x_{i}+O(|x'|^{m+1}))\overline{x_{d}\partial_{d}u}+O(|x'|^{2m})\partial_{i}\bar{u}$, $\overline{x_{d}\partial_{d}u}$ denotes the average of $x_{d}\partial_{d}u$ with regard to $x_{d}$ in $(-\varepsilon/2+h_{2},\varepsilon/2+h_{1})$. In light of the fact that $|x_{d}|\leq C(\varepsilon+\lambda_{0}|x'|^{m})$, we deduce from \eqref{U002} and \eqref{U001} that
\begin{align}\label{TQ001}
|F|\leq C(d)|x'|^{m-1}(\varepsilon+\lambda_{0}|x'|^{m}),\quad\mathrm{in}\; B_{1}'.
\end{align}

Since $\varphi$ is odd with respect to $x_{1}$ and the domain $\Omega=B_{5}\setminus\overline{D_{1}\cup D_{2}}$ is symmetric, then it follows from the elliptic theory that $u$ is odd in $x_{1}$ and $u$ is smooth. Then $\bar{u}$ is also odd in $x_{1}$ and $\bar{u}(0')=0$. Based on these facts, we use spherical harmonics to expand $\bar{u}$ as follows:
\begin{align}\label{VDAZ001}
\bar{u}(x')=U_{1,1}(r)Y_{1,1}(\xi)+\sum^{\infty}_{k=2}\sum^{N(k)}_{l=1}U_{k,l}(r)Y_{k,l}(\xi),\quad\mathrm{in}\;B_{1}'\setminus\{0'\},
\end{align}
where $\{Y_{k,l}\}_{k,l}$, which is an orthonormal basis of $L^{2}(\mathbb{S}^{d-2})$, consists of $k$-th degree normalized spherical harmonics, and $U_{k,l}\in C([0,1))\cap C^{\infty}((0,1))$ is given by $U_{k,l}=\int_{\mathbb{S}^{d-2}}\bar{u}(r,\xi)Y_{k,l}(\xi)d\xi$. In view of the fact that $\varepsilon+\lambda_{0}|x'|^{m}$ is independent of $\xi$ and $\bar{u}(0')=0$, we obtain that $U_{1,1}(0)=0$, and
\begin{align*}
LU_{1,1}:=\partial_{rr}U_{1,1}(r)+\left(\frac{d-2}{r}+\frac{m\lambda_{0}r^{m-1}}{\varepsilon+\lambda_{0}r^{m}}\right)\partial_{r}U_{1,1}(r)-\frac{d-2}{r^{2}}U_{1,1}(r)=H(r),
\end{align*}
for $0<r<1$, where
\begin{align*}
H(r)&=\int_{\mathbb{S}^{d-2}}\frac{(\mathrm{div}F)Y_{1,1}(\xi)}{\varepsilon+\lambda_{0}r^{m}}d\xi=\int_{\mathbb{S}^{d-2}}\frac{\partial_{r}F_{r}+\frac{1}{r}\nabla_{\xi}F_{\xi}}{\varepsilon+\lambda_{0}r^{m}}Y_{1,1}(\xi)d\xi\notag\\
&=\partial_{r}\left(\int_{\mathbb{S}^{d-2}}\frac{F_{r}Y_{1,1}}{\varepsilon+\lambda_{0}r^{m}}d\xi\right)+\int_{\mathbb{S}^{d-2}}\left(\frac{m\lambda_{0}r^{m-1}F_{r}Y_{1,1}}{(\varepsilon+\lambda_{0}r^{m})^{2}}-\frac{F_{\xi}\nabla_{\xi}Y_{1,1}}{r(\varepsilon+\lambda_{0}r^{m})}\right)d\xi\notag\\
&=:\partial_{r}A(r)+B(r),\quad\mathrm{in}\;(0,1),
\end{align*}
and $A(r),B(r)\in C^{1}([0,1))$. From \eqref{TQ001}, we know that
\begin{align}\label{AB001}
|A(r)|\leq C(d)r^{m-1},\quad\mathrm{and}\;|B(r)|\leq C(d)r^{m-2},\quad\mathrm{in}\;(0,1).
\end{align}

\noindent{\bf Step 2.} Proof of
\begin{align}\label{QWMM001}
U_{1,1}(r)=C_{1}(\varepsilon)g(r)+O(r^{m-1+\alpha}),\quad\mathrm{in}\;(0,1),
\end{align}
where $\alpha$ is defined by \eqref{degree}, $g$ is the solution of \eqref{AMR001}, $C_{1}(\varepsilon)$ is some constant satisfying that
\begin{align}\label{QWMM002}
C_{1}(\varepsilon)\geq\frac{1}{C_{2}},\quad\text{for some positive constant}\; C_{2}\; \text{independent of}\;\varepsilon.
\end{align}

Define $v:=gw$, where $g$ is the solution of \eqref{AMR001} with $g(0)=0$ and $g(1)=1$, and
\begin{align*}
w(r):=\int_{0}^{r}\frac{1}{g^{2}(s)s^{d-2}(\varepsilon+\lambda_{0}s^{m})}\int^{s}_{0}g(t)t^{d-2}(\varepsilon+\lambda_{0}t^{m})H(t)\,dtds,\quad\mathrm{in}\;(0,1).
\end{align*}
It then follows from a straightforward calculation that
\begin{align*}
Lv=L(gw)=gw''+\left(2g'+\Big(\frac{d-2}{r}+\frac{m\lambda_{0}r^{m-1}}{\varepsilon+\lambda_{0}r^{m}}\Big)g\right)w'=\frac{g}{G}(Gw')'=H,
\end{align*}
where $G=g^{2}r^{d-2}(\varepsilon+\lambda_{0}r^{m})$. In light of $g'>0$ and using \eqref{AB001}, we have
\begin{align*}
&\int^{s}_{0}g(t)t^{d-2}(\varepsilon+\lambda_{0}t^{m})H(t)dt\notag\\
&=\int^{s}_{0}g(t)t^{d-2}(\varepsilon+\lambda_{0}t^{m})A'(t)dt+O(1)g(s)s^{d+m-3}(\varepsilon+\lambda_{0}s^{m})\notag\\
&=-\int_{0}^{s}g'(t)t^{d-2}(\varepsilon+\lambda_{0}t^{m})A(t)dt+O(1)g(s)s^{d+m-3}(\varepsilon+\lambda_{0}s^{m})\notag\\
&=O(1)\left(s^{d+m-3}(\varepsilon+\lambda_{0}s^{m})\int^{s}_{0}g'(t)dt+g(s)s^{d+m-3}(\varepsilon+\lambda_{0}s^{m})\right)\notag\\
&=O(1)g(s)s^{d+m-3}(\varepsilon+\lambda_{0}s^{m}),
\end{align*}
which, together with \eqref{DAM001}, reads that
\begin{align*}
|v(r)|\leq Cg(r)\int^{r}_{0}\frac{s^{m-1}}{g(s)}\leq Cr^{m-1+\alpha}.
\end{align*}
Note that $U_{1,1}-v$ remains bounded and $L(U_{1,1}-v)=0$ for $r\in(0,1)$, it then follows from Lemma \ref{lemma006} that $U_{1,1}-v=C_{1}(\varepsilon)g$. That is, \eqref{QWMM001} holds.

We now prove \eqref{QWMM002}. Based on the assumed symmetric condition on the domain in Theorem \ref{thm002}, let $x=(r,\xi,x_{d})\in\mathbb{R}_{+}\times\mathbb{S}^{d-2}\times\mathbb{R}$ and then \eqref{con002} can be rewritten as follows:
\begin{align}\label{DWQ001}
\begin{cases}
\partial_{rr}u+\frac{d-2}{r}\partial_{r}u+\frac{1}{r^{2}}\Delta_{\mathbb{S}^{d-2}}u+\partial_{dd}u=0,&\mathrm{in}\;B_{5}\setminus\overline{D_{1}\cup D_{2}},\\
\frac{\partial u}{\partial\nu}=0,&\mathrm{on}\;\partial D_{i},\;i=1,2,\\
u=x_{1},&\mathrm{on}\;\partial B_{5}.
\end{cases}
\end{align}
Denote
\begin{align*}
\tilde{u}(r,x_{d}):=\int_{\mathbb{S}^{d-2}}u(r,\xi,x_{d})Y_{1,1}(\xi)d\xi.
\end{align*}
Due to the fact that $u$ is an odd function with respect to $x_{1}$, we get $\tilde{u}(0,x_{d})=0$ for any $x_{d}$. Then multiplying equation \eqref{DWQ001} by $Y_{1,1}(\xi)$ and utilizing integration by parts on $\mathbb{S}^{d-2}$, we obtain that $\tilde{u}(r,x_{d})$ verifies
\begin{align}\label{DWQ002}
\begin{cases}
\partial_{rr}\tilde{u}+\frac{d-2}{r}\partial_{r}\tilde{u}-\frac{d-2}{r^{2}}\tilde{u}+\partial_{dd}\tilde{u}=0,&\mathrm{in}\;\tilde{B}_{5}\setminus\overline{\tilde{D}_{1}\cup \tilde{D}_{2}},\\
\frac{\partial \tilde{u}}{\partial\nu}=0,&\mathrm{on}\;\partial \tilde{D}_{i},\;i=1,2,\\
\tilde{u}=0,&\mathrm{on}\;\{r=0\},\\
\tilde{u}=r,&\mathrm{on}\;\partial \tilde{B}_{5},
\end{cases}
\end{align}
where $\nu$ denotes the unit inner normal of $\partial\tilde{D_{i}}$, $i=1,2,$ and
\begin{align*}
\tilde{B}_{5}:=&\{(r,x_{d})\in\mathbb{R}_{+}\times\mathbb{R}\,|\,r^{2}+x_{d}^{2}<25\},\\
\tilde{D}_{i}:=&\{(r,x_{d})\in\mathbb{R}_{+}\times\mathbb{R}\,|\,r^{m}+|x_{d}+(-1)^{i}(1+\varepsilon/2)|^{m}<1\}.
\end{align*}
Observe that $\tilde{v}(r)=r$ verifies the first line of \eqref{DWQ002} with $\frac{\partial\tilde{v}}{\partial\nu}<0$ on $\partial\tilde{D}_{i}$, $i=1,2.$ Therefore, $r$ becomes a subsolution of \eqref{DWQ002} and we thus have $\tilde{u}\geq r$. This yields that
\begin{align*}
U_{1,1}(r)=\fint_{-\frac{\varepsilon}{2}+h_{2}}^{\frac{\varepsilon}{2}+h_{1}}\tilde{u}(r,x_{d})\,dx_{d}\geq r,
\end{align*}
which, together with \eqref{DAM001} and \eqref{QWMM001}, leads to that
\begin{align*}
r\leq U_{1,1}(r)=C_{1}(\varepsilon)g(r)+O(r^{m-1+\alpha})\leq C_{1}(\varepsilon)r^{\alpha}+\frac{1}{2}r,\quad\mathrm{in}\;(0,r_{0}],
\end{align*}
for some small $\varepsilon$-independent constant $r_{0}$. Then we get
\begin{align*}
C_{1}(\varepsilon)\geq\frac{1}{2}r_{0}^{1-\alpha}.
\end{align*}

\noindent{\bf Step 3.} Combining the results above, we give the proof of Theorem \ref{thm002}. Using \eqref{DAM001} and \eqref{QWMM001}--\eqref{QWMM002}, we obtain that for $\beta\geq\frac{2\alpha^{2}+\alpha(d+m-3)}{2\alpha+d-3}$,
\begin{align}\label{DYDA001}
U_{1,1}(r)\geq\frac{1}{C}g(r)-Cr^{m-1+\alpha}\geq\frac{1}{2C}r^{\beta}(\varepsilon+\lambda_{0}r^{m})^{\frac{\alpha-\beta}{m}},\quad\mathrm{in}\;(0,r_{0}],
\end{align}
where $r_{0}$ is a small positive $\varepsilon$-independent constant. From \eqref{VDAZ001} and \eqref{DYDA001}, we deduce
\begin{align*}
\left(\int_{\mathbb{S}^{d-2}}|\bar{u}(\sqrt[m]{\varepsilon},\xi)|^{2}d\xi\right)^{\frac{1}{2}}\geq|U_{1,1}(\sqrt[m]{\varepsilon})|\geq\frac{1}{C}\varepsilon^{\frac{\alpha}{m}},
\end{align*}
which reads that $|\bar{u}(\sqrt[m]{\varepsilon},\xi_{0})|\geq\frac{1}{C}\varepsilon^{\frac{\alpha}{m}}$ for some $\xi_{0}\in\mathbb{S}^{d-2}$. In view of the fact that $\bar{u}$ is the average of $u$ in the $x_{d}$ direction, we obtain
\begin{align*}
|u(\sqrt[m]{\varepsilon},\xi_{0},x_{d})|\geq\frac{1}{C}\varepsilon^{\frac{\alpha}{m}},\quad\text{for some}\;x_{d}\in(-\varepsilon/2+h_{2}(x'),\varepsilon/2+h_{1}(x')).
\end{align*}
This, together with the fact that $u(0)=0$, implies that \eqref{ARQZ001} holds. The proof is complete.

\end{proof}

\noindent{\bf{\large Acknowledgements.}}
The author would like to thank Prof. C.X. Miao for his constant encouragement and useful discussions. The author was partially supported by CPSF (2021M700358).

\bibliographystyle{plain}

\def\cprime{$'$}

\end{document}